\documentclass[12pt]{article}

\usepackage[english]{babel}
\usepackage{geometry} 
\usepackage{latexsym}
\usepackage{float}
\usepackage{url}

\usepackage{amsmath,amsthm,amssymb}
\usepackage{xcolor}

\usepackage[hyperfootnotes=false]{hyperref}

\usepackage{titlesec}
\titleformat{\paragraph}
{\normalfont\normalsize\bfseries}{\theparagraph}{1em}{}
\titlespacing*{\paragraph}
{0pt}{3.25ex plus 1ex minus .2ex}{1.5ex plus .2ex}

\def\titlerunning#1{\gdef\titrun{#1}}
\makeatletter
\def\author#1{\gdef\autrun{\def\and{\unskip, }#1}\gdef\@author{#1}}
\def\address#1{{\def\and{\\\hspace*{18pt}}\renewcommand{\thefootnote}{}%
\footnote {#1}}%
\markboth{\autrun}{\titrun}}
\makeatother
\def\email#1{\hspace*{4pt}{\em e-mail}: #1}

\newtheorem{thm}{Theorem}[section]
\newtheorem{prop}[thm]{Proposition}
\newtheorem{lemma}[thm]{Lemma}
\newtheorem{cor}[thm]{Corollary}
\theoremstyle{definition}
\newtheorem{rem}[thm]{Remark}
\newtheorem{defi}[thm]{Definition}
\newtheorem{example}[thm]{Example}
\newtheorem{result}[thm]{Result}

\newtheorem{notation}[thm]{Notation}

\newtheoremstyle{dotless}{}{}{\itshape}{}{\bfseries}{}{ }{}

\newcommand{\E}{{\mathcal E}}
\newcommand{\F}{{\mathbb F}}
\DeclareMathOperator\SRG{SRG}
\DeclareMathOperator\RG{RG}
\DeclareMathOperator\DDG{DDG}
\newcommand{\comments}[1]{}
\newcommand{\gs}[3]{\genfrac{[}{]}{0pt}{}{#1}{#2}_{#3}}
\def\F{\mathbb{F}}

\def\N{\mathbb{N}}

\def\S{\mathcal{S}}
\def\cQ{\mathcal{Q}}
\def\cR{\mathcal{R}}

\DeclareMathOperator{\PG}{PG}
\DeclareMathOperator{\GL}{GL}
\DeclareMathOperator{\PGL}{PGL}

\begin{document}

\titlerunning{}
\title{$q$-Analogs of divisible design graphs and Deza graphs}

\author{Dean Crnkovi\' c \and Maarten De Boeck \and Francesco Pavese \and Andrea \v Svob}
\date{}

\maketitle

\address{D. Crnkovi\'{c}, A. \v Svob: Faculty of Mathematics, University of Rijeka, Croatia;
\email{\{deanc,asvob\} @math.uniri.hr}
\and
M. De Boeck: Department of Mathematical Sciences, University of Memphis, TN, USA; 
\email{mdeboeck@memphis.edu} 
\and
F. Pavese: Dipartimento di Meccanica, Matematica e Management, Politecnico di Bari, Italy;
\email{francesco.pavese@poliba.it}
}

\begin{abstract}
Divisible design graphs were introduced in 2011 by Haemers, Kharaghani and Meulenberg. In this paper, we introduce the notion of $q$-analogs of divisible design graphs and show that all $q$-analogs of divisible design graphs come from spreads, and are actually $q$-analogs of strongly regular graphs.
\par Deza graphs were introduced by Erickson, Fernando, Haemers and Hardy in 1999. In this paper, we introduce $q$-analogs of Deza graphs. Further, we determine possible parameters, give examples of $q$-analogs of Deza graphs and characterize all non-strongly regular $q$-analogs of Deza graphs with the smallest parameters.
\end{abstract}

\bigskip



\section{Introduction}\label{intro}

A graph $\Gamma$ can be interpreted as a design by taking the vertices of $\Gamma$ as points, and the neighborhoods of the vertices as blocks. Such a design is called the \emph{neighborhood design} of $\Gamma$. The adjacency matrix of $\Gamma$ is the incidence matrix of its neighborhood design.

A $k$-regular graph on $v$ vertices with the property that any two distinct vertices have exactly $\lambda$ common neighbors is called a $(v,k, \lambda)$-graph (see \cite{rudvalis}) or \emph{edge-regular} with parameters $(v,k,\lambda)$. The neighborhood design of a $(v,k, \lambda)$-graph is a symmetric $(v,k, \lambda)$ design. Haemers, Kharaghani and Meulenberg have defined divisible design graphs (DDGs for short) as a generalization of $(v,k, \lambda)$-graphs (see \cite{ddg}).

An incidence structure with $v$ points and the constant block size $k$ is a (group) \emph{divisible design} with parameters $(v,k, \lambda_1, \lambda_2, m,n)$ whenever the point set can be partitioned into $m$ classes of size $n$, such that two points from the same class are incident with exactly $\lambda_1$ common blocks, and two points from different classes are incident with exactly $\lambda_2$ common blocks. A divisible design $D$ is said to be \emph{symmetric} (or to have the dual property) if the dual of $D$ is a divisible design with the same parameters as $D$.

A $k$-regular graph with $v$ vertices is a \emph{divisible design graph (DDG)} with parameters $(v,k, \lambda_1, \lambda_2, m,n)$ if the vertex set can be partitioned into $m$ classes of size $n$ such that two vertices of the same class have precisely $\lambda_1$ common neighbors, and two vertices from different classes have precisely $\lambda_2$ common neighbors. Note that a DDG with $m = 1$, or $n = 1$, or $\lambda_1 = \lambda_2$ is a $(v, k, \lambda)$-graph. If this is the case, we call the DDG \emph{improper}, otherwise it is called \emph{proper}. The definition of a DDG yields the following theorem (see \cite{ddg}).

\begin{thm}
    If $D$ is a divisible design graph with parameters $(v,k, \lambda_1, \lambda_2, m,n)$ then its neighborhood design is a symmetric divisible design $(v,k, \lambda_1, \lambda_2, m,n)$.
\end{thm} 

Conversely, a symmetric divisible design with a polarity with no absolute points is the neighborhood design of a DDG (see \cite{ddg}). We also have the following counting result (\cite[Section 2]{ddg}).

\begin{lemma}\label{lem:ddgcounting}
    For a divisible design graph with parameters $(v,k, \lambda_1, \lambda_2, m,n)$ we have $k^{2}=k+\lambda_{1}(n-1)+\lambda_{2}n(m-1)$ and $v=mn$.
\end{lemma}

Now we look at a second class of regular graphs, generalizing the DDG's. \emph{Deza graphs} were introduced by Erickson, Fernando, Haemers and Hardy in \cite{deza}, as a generalization of strongly regular graphs. Let $v$, $k$, $b$ and $a$ be integers such that $0 \leq a < b \leq k < v$. A $k$-regular graph on $v$ vertices is called a Deza graph with parameters $(v, k, b, a)$ if any two distinct vertices have either $a$ or $b$ common neighbors. Recently, Deza graphs attended a lot of attention of researchers, see e.g. \cite{spectra-deza, Deza-Haemers,Deza-Goryainov,Deza-Kabanov}.

Given a vertex of a Deza graph we know with how many other vertices it has precisely $b$ (or precisely $a$) common neighbors (\cite[Proposition 1.1]{deza}).

\begin{lemma}\label{lem:dezacounting}
    For a Deza graph $\Gamma=(V,E)$ with parameters $(v,k,b,a)$ with $b>a$, and for $u\in V$ we have that
    \begin{align*}
        |\{v\in V\mid |N_{\Gamma}(u)\cap N_{\Gamma}(v)|=b\}|=\frac{k(k-1)-a(v-1)}{b-a}.
    \end{align*}
    In this statement, we can interchange the roles of $b$ and $a$. Note that the result is independent of the vertex $u$.
\end{lemma}

In \cite{delsarte}, Delsarte introduced the notion of \emph{$q$-analogs of designs} or \emph{subspace design}. The interest in this subject has been renewed recently due to applications of $q$-analog of designs to network coding and distributed storage, (see \cite{mario-subspace}). Let $\mathbb{F}_q^v$ be a vector space of dimension $v$ over the finite field $\mathbb{F}_q$. A subspace of vector space $\mathbb{F}_q^v$ of dimension $k$ will be called a $k$-subspace, or $k$-space. A (simple) $t$-$(v,k,\lambda;q)$ subspace design consists of a set $B$ of $k$-subspaces of $\mathbb{F}_q^v$, called blocks, such that each $t$-subspace of $\mathbb{F}_q^v$ is contained in exactly $\lambda$ blocks. Subspace designs are also called designs over finite fields, designs in vector spaces, designs in the $q$-Johnson scheme, and geometric designs (see e.g. \cite{spyros}). The set of all $k$-subspaces of $\mathbb{F}_q^v$ is always a design, called the \emph{trivial design}. Recently, in \cite{buratti} the authors introduced $q$-analogues of group divisible designs and in \cite{buratti2} gave more results on this topic.

In \cite{qsrg}, the following definitions of $q$-ary graph, regular $q$-ary graph and strongly regular $q$-ary graph are given. 
We use the notation $\gs{W}{s}{}$ for the set of $s$-dimensional subspaces of the vector space $W$.
\begin{defi}\label{def:qgraph}
	A \emph{$q$-ary graph} $\E$ is a pair $(V,E)$ where $V=\gs{W}{1}{}$ and $E\subseteq\gs{W}{2}{}$, with $W$ a vector space over $\F_{q}$. The elements of $V$ and $E$ are called \emph{vertices} and \emph{edges}, respectively. If $\langle X,Y\rangle$ is an edge, then we say that the vertices $X$ and $Y$ are \emph{adjacent}.
\end{defi}

A $q$-ary graph is thus a hypergraph whose vertices are the vector lines of a vector space $W$ over $\F_{q}$, and whose edges correspond to vector planes of $W$. A $q$-ary graph over $W$ is naturally embedded in the projective geometry $\PG(W)$. Its vertices are the points of the projective space $\PG(W)$ and the edge set is a subset of the line set. Based on a $q$-ary graph $\E$ we can define a corresponding simple graph whose vertices are the vertices of $\E$ and with two vertices adjacent if they are adjacent in $\E$ (see also \cite{qsrg}). A $q$-ary graph will be called (dis)connected if its corresponding simple graph is (dis)connected.

\begin{defi}
	For a vertex $X$ of a $q$-ary graph $\E$ we define its \emph{neighborhood} $N_{\E}(X)$ as the set of all vertices adjacent to it, and itself.
\end{defi}

\begin{defi}
	A $q$-ary graph $\E$ is \emph{$k$-regular} if for any vertex $X$ its neighborhood $N_{\E}(X)$ is a $(k+1)$-dimensional subspace. The set of all $k$-regular $q$-ary graphs with underlying vector space $\F_q^v$ will be denoted by $\RG(v,k;q)$.
\end{defi}

It is a basic result in graph theory that a $k$-regular graph has $\frac{vk}{2}$ edges. In particular, $vk$ is even. For regular $q$-graphs a similar result holds. For a prime power $q$ and an integer $i$ we denote $\frac{q^{i}-1}{q-1}$ by $[i]_{q}$; this is the number of vector lines in $\F_{q}^{i}$.

\begin{lemma}\label{lem:numberedges}
    A $k$-regular $q$-ary graph $\E$ has $\frac{[v]_q [k]_q}{[2]_q}$ edges.
\end{lemma}
\begin{proof}
    Count the incident vertex-edge pairs.
\end{proof}

\begin{defi}
    A regular $q$-ary graph $\E\in\RG(v,k;q)$ is a strongly regular $q$-ary graph with parameters $(v,k,\lambda,\mu;q)$ if and only if for all vertices $X,Y$, with $X\ne Y$, we have that
	\[
	   |(N_{\E}(X)\cap N_{\E}(Y))\setminus\{X,Y\}|=\begin{cases}
	       \lambda&\text{if $X$ and $Y$ are adjacent,}\\
	       \mu&\text{otherwise.}
	   \end{cases}
	\]
    The set of all strongly regular $q$-ary graphs with these parameters $(v,k,\lambda,\mu)$ is denoted by $\SRG(v,k,\lambda,\mu;q)$.
\end{defi}

In \cite{qsrg}, the authors show the following classification:

\begin{result}{\cite{qsrg}}\label{th:srgclass}
	If $\E\in SRG(v,k,\lambda,\mu;q)$, then
    \begin{enumerate}
        \item $\E$ is the disjoint union of $\frac{q^{v}-1}{q^{k+1}-1}$ complete $q$-graphs on the elements of a $(k+1)$-spread, or
        \item $\E$ has parameters $(v,v-2,\mu-2,\mu;q)$ with $\mu=\frac{q^{v-2}-1}{q-1}$ and arises from a symplectic polarity $\varphi$ on the vector space $\F_{q}^{v}$, $v$ even, in the following way: $\E$ is the $q$-ary graph on $\F_{q}^{v}$ whose edges are the totally isotropic vector planes of $\varphi$.
    \end{enumerate}
\end{result}

In Section \ref{sec:ddg}, we will now introduce and discuss the $q$-analog for divisible design graphs, and in Section \ref{sec:deza}, we will discuss the $q$-analog for Deza graphs.

\section{\texorpdfstring{$q$}{q}-Analogs of divisible design graphs}\label{sec:ddg}

In this section, we introduce the notion of a $q$-analog of divisible design graphs. Throughout the section $v$, $k$, and $n$ will denote positive integers and $\lambda_1$ and $\lambda_2$ non-negative integers. We first recall the definition of a vector space partition or spread.

\begin{defi}
    A set $\mathcal{G}$ of $n$-spaces of a vector space $W$ forms a \emph{vector space partition into subspaces of dimension $n$} or \emph{$n$-spread} if every vector in $W$ different from the zero vector, is in contained in precisely one of the $n$-spaces in $\mathcal{G}$.
\end{defi}

It is immediate that an $n$-spread of $\F_{q}^{v}$ has $\frac{q^{v}-1}{q^{n}-1}$ elements.

\begin{rem}
    Vector space partitions have often been discussed via their interpretation in a projective geometry context. The names `spreads' arises from this context. It is classic result by Segre that $\F_{q}^{v}$ admits a vector space partition into subspaces of dimension $n$ if and only if $n\mid v$ (\cite{segre}).
\end{rem}

\begin{defi}\label{qddg}
    Let $\mathcal{G}$ be an $n$-spread of $\F_q^v$. A $k$-regular $q$-ary graph $\E\in\RG(v,k;q)$ is a \emph{$q$-ary divisible design graph} (shortly $q$-DDG) with respect to $\mathcal{G}$ if for all vertices $X$ and $Y$ with $X\neq Y$ we have that
    \[
        |(N_{\E}(X)\cap N_{\E}(Y))\setminus\{X,Y\}|=
        \begin{cases}
            \lambda_1&\text{if a $G\in\mathcal{G}$ exists such that $X,Y\subseteq G$,}\\
            \lambda_2&\text{otherwise.}
	    \end{cases}
	\]
    In this case $\mathcal{E}$ is said to have parameters $(v,k,\lambda_1,\lambda_2,n;q)$. The set of all $q$-ary divisible design graphs (for some $n$-spread) with parameters $(v,k,\lambda_1,\lambda_2,n;q)$ is denoted by $\DDG(v,k,\lambda_1,\lambda_2,m,n;q)$.
\end{defi}

\begin{rem}
    Note that a $q$-DDG with $n=v$ or $\lambda_1 = \lambda_2$ is a $q$-SRG with parameters $(v, k, \lambda_1,\lambda_1)$.
\end{rem}

In the following theorem, we show that a $\DDG(v,k,\lambda_1,\lambda_2,n;q)$ yields a classical DDG.

\begin{thm}\label{th:q-DDG->DDG}
    If a $q$-DDG with parameters $(v,k,\lambda_1,\lambda_2,n;q)$ exists, then a classical DDG with parameters $\left([v]_q,[k+1]_q-1,\lambda_1,\lambda_2,[n]_q\right)$ exists.
\end{thm}
\begin{proof}
    Consider a $q$-DDG $\E$ with parameters $\left(v,k,\lambda_1,\lambda_2,n;q\right)$ with respect to a vector space partition $\mathcal{G}$. Now define a (simple) graph $\Gamma$ with as vertex set the vertices of $\E$ (i.e.~the vector lines of $\F_{q}^{v}$) and such that two vertices of $\Gamma$ are adjacent if and only if the corresponding 1-subspaces belong to an edge of $\E$ (every hypergraph edge is replaced by edges consisting of two vertices in it). It is straightforward to check that $\Gamma$ is a DDG with parameters $([v]_q,[k+1]_q-1,\lambda_1,\lambda_2,[n]_q)$, where the classes correspond to the vertex sets of the elements of $\mathcal{G}$.
\end{proof}

\begin{example}
	The following are trivial examples of $q$-DDG's.
	\begin{enumerate}
		\item For any $n$-spread, the complete graph is a $q$-DDG with parameters $(v,v-1,\lambda,\lambda,n;q)$, with $\lambda=\frac{q^{v}-1}{q-1}-2$ and $n\mid v$.
		\item For any $n$-spread $S$, the disjoint union of the complete graphs on the elements of $S$ is a $q$-DDG with parameters $(v,n-1,\lambda,0,n;q)$, with $\lambda=\frac{q^{n}-1}{q-1}-2$ and $n\mid v$.
		\item For any $n$-spread the empty graph is a $q$-DDG with parameters $(v,0,0,0,n;q)$, with $n\mid v$.
	\end{enumerate}
\end{example}

Recall Lemma \ref{lem:ddgcounting} that for classical DDG's we have $k^{2}=k+\lambda_{1}(n-1)+\lambda_{2}n(m-1)$ with $v=mn$. Or in other words, $k^{2}=k+\lambda_{1}(n-1)+\lambda_{2}(v-n)$.

\begin{thm}\label{th:ddgstandardequation}
	If $\E$ is a q-DDG with parameters $(v,k,\lambda_{1},\lambda_{2},n;q)$, then
	\begin{align*}
		\left([k+1]_{q}-1\right)^{2}=\left([k+1]_{q}-1\right)+\lambda_{1}\left([n]_{q}-1\right)+\lambda_{2}\left([v]_{q}-[n]_{q}\right).
	\end{align*}
\end{thm}
\begin{proof}
	Let $\mathcal{G}$ be the $n$-spread with respect to which $\E$ is defined. We count the triples $(X_{1},X_{2},Y)$ with $X_{1},X_{2},Y$ vertices of $\E$, such that $Y\in N_{\E}(X_{1})\cap N_{\E}(X_{2})$ and $X_{1}\neq Y \neq X_{2}$. On the one hand, there are $[v]_{q}$ choices for $Y$. Given $Y$, we have $[k+1]_{q}-1$ choices for both $X_{1}$ and $X_{2}$. 
    \par On the other hand, there are $[v]_{q}$ choices for $X_{1}$. Given $X_1$, we have three options for $X_{2}$; let $G$ bet the element of $\mathcal{G}$ in which $X_1$ is contained. If $X_{1}=X_{2}$, then we have $[k+1]_{q}-1$ choices for $Y$. If $X_{1}\neq X_{2}$ and $X_{2}$ is contained in $G$, then there are $\lambda_1$ choices for $Y$. Note that there are $[n]_{q}-1$ vertices $X_2$ in this case. If $X_{2}$ is not contained in $G$, then there are $\lambda_2$ choices for $Y$. Note that there are $[v]_{q}-[n]_q$ vertices $X_2$ not in $G$.
    \par We find that
    \begin{align*}
		[v]_q\left([k+1]_{q}-1\right)^{2}=[v]_q\left(\left([k+1]_{q}-1\right)+\lambda_{1}\left([n]_{q}-1\right)+\lambda_{2}\left([v]_{q}-[n]_{q}\right)\right)\:,
	\end{align*}
    from which the statement immediately follows.
\end{proof}

We first look at a particular kind of $q$-DDG's.

\begin{lemma}\label{lem:lambda=mu}
	If $\E$ is a q-DDG with parameters $(v,k,\lambda_{1},\lambda_{2},v;q)$, then $\E$ is trivial, namely complete or empty.
\end{lemma}
\begin{proof}
	It immediately follows that $\E\in \SRG(v,k,\lambda_{1},\lambda_{1};q)$. From Theorem \ref{th:srgclass} the statement follows (given the parameters of these graphs).	
\end{proof}

Note that $\lambda_{2}$ is actually not defined in the case described by the previous lemma.

\begin{lemma}\label{lem:ddgnumbers1}
	If $\E\in \DDG(v,k,\lambda_{1},\lambda_{2},n;q)$, then
	\begin{itemize}
		\item any two vertices in the same spread element are adjacent, or
		\item no two vertices in the same spread element are adjacent, or
		\item $\lambda_{1}=1$ and $q=2$.
	\end{itemize}
\end{lemma}
\begin{proof}
	We assume neither the first nor the second possibility in the statement of the lemma is fulfilled. So, there are non-adjacent vertices $X,X'$, with $X'\neq X$, belonging to the same spread element. Also there are adjacent vertices $Y,Y'$ belonging to the same spread element.
    \par Since $N_{\E}(X)$ and $N_{\E}(X')$ are subspaces, also $N_{\E}(X)\cap N_{\E}(X')$ is a subspace. Note that $N_{\E}(X)\cap N_{\E}(X')$ contains the 2-subspace $\langle X,X'\rangle$. It follows that $\lambda_{1}=\frac{q^{a}-1}{q-1}-2$ for some integer $a\geq 2$. Since $N_{\E}(Y)$ and $N_{\E}(Y')$ are subspaces, also $N_{\E}(Y)\cap N_{\E}(Y')$ is a subspace. Note that $N_{\E}(Y)\cap N_{\E}(Y')$ contains neither $Y$ nor $Y'$. It follows that $\lambda_{1}=\frac{q^{b}-1}{q-1}$ for some integer $b\geq 0$. We find that
	\begin{align*}
	 	\frac{q^{a}-1}{q-1}-2=\frac{q^{b}-1}{q-1} \quad\Leftrightarrow\quad q^{a}-q^{b}=2(q-1),
	\end{align*}
	which implies that $a>b$. We rewrite the previous expression once more and find that
	\begin{align*}
	q^{b}\frac{q^{a-b}-1}{q-1}=2.
	\end{align*}
	It follows that $a=2$, $b=1$ and $q=2$. Hence, $\lambda_{1}=1$.
\end{proof}

Similarly we can also prove.

\begin{lemma}\label{lem:ddgnumbers2}
	If $\E\in\DDG(v,k,\lambda_{1},\lambda_{2},n;q)$, then
	\begin{itemize}
		\item any two vertices in a different spread element are adjacent, or
		\item no two vertices in a different spread element are adjacent, or
		\item $\lambda_{2}=1$ and $q=2$.
	\end{itemize}
\end{lemma}

\begin{thm}
	If $\E$ is a q-DDG with parameters $(v,k,\lambda_{1},\lambda_{2},n;q)$ with $2\leq n<v$, then $\E$ is trivial.	
\end{thm}
\begin{proof}
	By Lemma \ref{lem:ddgnumbers2} we can distinguish between three cases. If any two vertices in a different spread element are adjacent, then the neighborhood of each vertex must equal $\F^{v}_{q}$, since the whole space is the only subspace containing all spread elements but one. So, $\E$ is the complete graph.
	\par If no two vertices in a different spread element are adjacent, $\E$ is the disjoint union of $q$-graphs $\E_{1},\dots,\E_{m}$, with $m=\frac{q^{v}-1}{q^{n}-1}$, each defined on an element of the spread. Each of the $\E_{i}$ is a $q$-graph fulfilling the conditions of Lemma \ref{lem:lambda=mu}, so is either complete or empty. However, if one of them is the empty graph, then $\lambda_{1}=0$, so they all must be the empty graph. Hence, either $\E$ is the empty graph, or the disjoint union of complete graphs on the elements of the spread.
	\par Finally we look at the case with $\lambda_{2}=1$ and $q=2$. From Theorem \ref{th:ddgstandardequation} we find that
	\begin{align}\label{eq:geval3}
		&\left(2^{k+1}-2\right)^{2}=\left(2^{k+1}-2\right)+\lambda_{1}\left(2^{n}-2\right)+\left(2^{v}-2^{n}\right)\nonumber\\
		\Leftrightarrow\quad&2^{2k+1}-2^{k+2}-2^{k}+3=\lambda_{1}\left(2^{n-1}-1\right)+2^{v-1}-2^{n-1}\;.
	\end{align}
	We now distinguish two cases based on the results of Lemma \ref{lem:ddgnumbers1}, combining the second and third case of this lemma. We can assume without loss of generality that $v\geq k+2$ since for $v=k+1$ we get the complete graph, and also that $k\geq1$ since for $k=0$ we get the empty graph.
	\begin{itemize}
		\item If any two vertices in the same spread element are adjacent, then $\lambda_{1}=(2^{a}-1)-2$ for some integer $a$ with $n\leq a\leq k+1$, since $N_{\E}(X)\cap N_{\E}(X')$ is a subspace containing both $X$ and $X'$, for any vertices $X,X'$ in the same spread element. Moreover, if $k+1=n$, then two vertices from different spread elements cannot be adjacent, contradicting $\lambda_{2}=1$. So $k\geq n$. Substituting $\lambda_{1}=2^{a}-3$ in Equation \ref{eq:geval3} we find
		\begin{align}\label{eq:geval3.1}
			2^{2k+1}-2^{k+2}-2^{k}=2^{v-1}+2^{a+n-1}-2^{a}-2^{n+1}\;.
		\end{align}
		We distinguish between three subcases.
		\begin{itemize}
			\item If $a=n$ then Equation \eqref{eq:geval3.1} reduces to
			\begin{align*}
				2^{2k+1}-2^{k+2}-2^{k}=2^{v-1}+2^{2n-1}-2^{n}-2^{n+1}\;.
			\end{align*}
			From reducing modulo $2^{n+1}$ it follows that $k=n$, and thus that $2^{2k}+2^{2k-1}=2^{v-1}+2^{k+1}$. We find that $(v,k)=(5,2)$. However, then $n\nmid v$, a contradiction.
			\item If $a=n+1$ then Equation \eqref{eq:geval3.1} reduces to
			\begin{align*}
				2^{2k+1}-2^{k+2}-2^{k}=2^{v-1}+2^{2n}-2^{n+2}\;.
			\end{align*}
			Recall that $n\geq2$, and note that for $n=2$ we immediately have a contradiction. So, $n\geq3$. It follows that $k=n+2$ or $v-1=n+2$, since $v\geq k+2$. If $k=n+2$, then $2^{2k+1}-2^{2k-4}-2^{k+2}=2^{v-1}$, a contradiction. If $v-1=n+2$, then $2^{2k+1}-2^{k+2}-2^{k}=2^{2n}$, also a contradiction.
			\item If $a\geq n+2$, then reducing Equation \eqref{eq:geval3.1} modulo $2^{a}$ learns that $k=n+1$; recall that $v-1\geq k+1\geq a$. Equation \eqref{eq:geval3.1} reduces to
			\begin{align*}
				2^{2k+1}-2^{k+2}=2^{v-1}+2^{a+k-2}-2^{a}\;.
			\end{align*}
			It follows that $2^{a+k-2}-2^{a}\equiv0\pmod{2^{k+1}}$. Since $a\geq n+2$ we have that $a\geq4$, so $k\geq3$, hence $a+k-2>a$. Therefore it follows from $2^{a+k-2}-2^{a}\equiv0\pmod{2^{k+1}}$ that $a=k+1$. We find that $2^{2k}+2^{2k-1}=2^{v-1}+2^{k+1}$. Exactly as in the first case we find a contradiction.
		\end{itemize}
		\item If no two vertices in the same spread element are adjacent, then $\lambda_{1}=2^{a}-1$ for some integer $a$ with $a\geq 0$, since $N_{E}(X)\cap N_{E}(X')$ is a subspace, for any vertices $X,X'$. If $\lambda_{1}=1$, then $\lambda_{1}=2^{a}-1$ for $a=1$. 
		Substituting $\lambda_{1}=2^{a}-1$ in Equation \ref{eq:geval3} we find
		\begin{align}\label{eq:geval3.2}
			2^{2k+1}-2^{k+2}-2^{k}+2=2^{v-1}+2^{a+n-1}-2^{a}-2^{n}\;.
		\end{align}
		Since $n\geq2$ and $v-1\geq k+1\geq2$, we must have that $a=1$ or $k=1$. We distinguish between these two subcases.
		\begin{itemize}
			\item If $a=1$, then Equation \eqref{eq:geval3.2} reduces to
			\begin{align*}
				2^{2k+1}-2^{k+2}-2^{k}+4=2^{v-1}\;.
			\end{align*}
			Recall that $v$ is a non-trivial multiple of $n$, so $v\geq4$. Hence, reducing the above expression modulo 8, learns that $k=2$. Then, $v=5$ a prime. However $n$ is a nontrivial divisor of $v$, a contradiction.
			\item If $k=1$, then Equation \eqref{eq:geval3.2} reduces to
			\begin{align*}
				0=2^{v-1}+2^{a+n-1}-2^{a}-2^{n}\;,
			\end{align*}
			a contradiction since $2^{n}\leq 2^{v-1}$ and $2^{a}<2^{a+n-1}$.
		\end{itemize}
	\end{itemize}
	We conclude that in each of the cases we find no examples, or only the trivial examples.
\end{proof}

\section{\texorpdfstring{$q$}{q}-Analogs of Deza graphs}\label{sec:deza}

In the following definition of $q$-analogs of Deza graphs, we generalize the notion of $q$-analogs of strongly regular graphs in a similar way as Deza graphs generalize strongly regular graphs.

\begin{defi}\label{qDeza}
A $k$-regular $q$-ary graph $\E\in\RG(v,k;q)$ is a \emph{$q$-ary Deza graph} (shortly \emph{$q$-Deza graph}) with parameters $(v,k,b,a;q)$ if for all vertices $X,Y$ with $X\neq Y$, it holds that $|(N_{\E}(X)\cap N_{\E}(Y))\setminus\{X,Y\}| \in \{ a,b \}$, where $a \le b$.
\end{defi}

\comments{
\begin{example}
	The following are trivial examples of qDDG's.
	\begin{enumerate}
		\item For any $n$-spread the complete graph is a qDDG with parameters $(v,v-1,\lambda,\lambda,n;q)$, with $\lambda=\frac{q^{v}-1}{q-1}-2$ and $n\mid v$.
		\item For any $n$-spread $S$, the disjoint union of the complete graphs on the elements of $S$ is a qDDG with parameters $(v,n-1,\lambda,0,n;q)$, with $\lambda=\frac{q^{n}-1}{q-1}-2$ and $n\mid v$.
		\item For any $n$-spread the empty graph is a qDDG with parameters $(v,0,0,0,n;q)$, with $n\mid v$.
	\end{enumerate}
\end{example}}

We now prove an analog for Lemma \ref{lem:dezacounting}.

\begin{thm}\label{th:dezastandardequation}
	Let $\E=(V,E)$ be a q-Deza graph with parameters $(v,k,\lambda_{1},\lambda_{2};q)$, with $\lambda_{1}\neq\lambda_{2}$. For any vertex $X$ we have
	\begin{align*}
		\left|\left\{Y\in V\mid |N_{\E}(X)\cap N_{\E}(Y)|=\lambda_{1}\right\}\right|=q\frac{\left(q^{k+1}-2q+1\right)\left(q^{k}-1\right)-\lambda_{2}\left(q^{v-1}-1\right)(q-1)}{(\lambda_{1}-\lambda_{2})(q-1)^2}\;.
	\end{align*}
	In this statement, we can interchange the roles of $\lambda_{1}$ and $\lambda_{2}$. Note that the result is independent of the vertex $X$.
\end{thm}
\begin{proof}
	We denote $|\{Y\in V\mid |N_{\E}(X)\cap N_{\E}(Y)|=\lambda_{i}\}|$ by $n_{i}$. It follows immediately from the definition of a $q$-ary Deza graph that
	\begin{align}\label{eq:Deza1}
		n_{1}+n_{2}=[v]_{q}-1\;.
	\end{align}
	Counting the triples $(Y_{1},Z,Y_{2})$ with $Y_{1},Y_{2},Z$ vertices of $\Gamma$ and $Z\in N_{\Gamma}(Y_{1})\cap N_{\Gamma}(Y_{2})$ and $Y_{1}\neq Z\neq Y_{2}$ in two ways, we find that
	\begin{align}\label{eq:Deza2}
		\left([k+1]_{q}-1\right)^{2}=\left([k+1]_{q}-1\right)+\lambda_{1}n_{1}+\lambda_{2}n_{2}.
	\end{align}
	Solving the system of Equations \eqref{eq:Deza1} and \eqref{eq:Deza2} for $n_{1}$ and $n_{2}$ gives
	\begin{align*}
		n_{1}&=\frac{([k+1]_{q}-1)([k+1]_{q}-2)-\lambda_{2}([v]_{q}-1)}{\lambda_{1}-\lambda_{2}}\\
        &=q\frac{(q^{k+1}-2q+1)(q^{k}-1)-\lambda_{2}(q^{v-1}-1)(q-1)}{(\lambda_{1}-\lambda_{2})(q-1)^2}\\
		n_{2}&=\frac{([k+1]_{q}-1)([k+1]_{q}-2)-\lambda_{1}([v]_{q}-1)}{\lambda_{2}-\lambda_{1}}\\
        &=q\frac{(q^{k+1}-2q+1)(q^{k}-1)-\lambda_{1}(q^{v-1}-1)(q-1)}{(\lambda_{2}-\lambda_{1})(q-1)^2}
	\end{align*}
	and the statement directly follows. 
\end{proof}

\begin{notation}
	For a $q$-Deza graph $\E=(V,E)$ with parameters $(v,k,\lambda_1,\lambda_2;q)$, we denote $\left|\left\{Y\in V\mid |N_{\E}(X)\cap N_{\E}(Y)|=\lambda_{i}\right\}\right|$, with $X$ an arbitrary vertex, by $n_{\Gamma}(\lambda_{i})$, $i=1,2$. Theorem \ref{th:dezastandardequation} shows that $n_{\Gamma}(\lambda_{i})$ is well-defined as it is indeed independent of the choice for $X$.
\end{notation}

The following theorem can be proven in a similar way as Theorem \ref{th:q-DDG->DDG}.

\begin{thm} \label{th:q-Deza->Deza}
If a $q$-ary Deza graph with parameters $(v,k,b,a;q)$ exist, then a classical Deza graph with parameters $([v]_q,[k+1]-1,b,a)$ exists.
\end{thm}

It is immediate that every $q$-ary strongly regular graph is a $q$-Deza graph. But unlike the situation for $q$-DDG's, where every $q$-DDG is actually a $q$-ary strongly regular graph, we can find examples of $q$-Deza graphs that are not strongly regular. We first recall the definition of a generalized hexagon, which is a special case of a generalized polygon. For background on generalized polygons we refer to \cite{bookvanmaldeghem}.

\begin{defi}
    A \emph{generalized hexagon} is a point-line geometry $(\mathcal{P},\mathcal{L})$ such that its incidence graph has diameter 6 and girth 12. The incidence graph is the bipartite graph with vertex set $\mathcal{P}\cup\mathcal{L}$ where $p\in\mathcal{P}$ and $\ell\in\mathcal{L}$ are adjacent if $p$ and $\ell$ are incident in the geometry. A generalized hexagon $\mathcal{H}$ has order $(s,t)$ if any line has $s+1$ points and through any point there are $t+1$ lines.
\end{defi}

\begin{example}\label{ex:DezasplitCayleyhexagon}
    There are up to isomorphism two generalized hexagons of order $(2,2)$, namely the split Cayley hexagon $\mathcal{H}(2)$, and its dual (see \cite{cohentits}). In particular, $\mathcal{H}(2)$ has a unique, up to projectivity, {\em regular embedding} in $\PG(5, 2)$, where $\mathcal{P}$ consists of all the points of $\PG(5, 2)$, $\mathcal{L}$ of a set of $63$ lines of $\PG(5, 2)$ such that
    \begin{itemize}
        \item the lines of $\mathcal{L}$ through a point are coplanar;
        \item the points at distance less than $4$ from any given point are contained in a hyperplane of $\PG(5, 2)$.
    \end{itemize}
For further information on $\mathcal{H}(2)$ the reader is referred to \cite{TV} and references therein. The set $\mathcal{L}$ of lines of $\mathcal{H}(2)$ can be interpreted as $q$-Deza graph. Indeed, let $P$ be a point of $\PG(5, 2)$ then there are three lines of $\mathcal{L}$ through $P$ and these lines are coplanar. Moreover, if $R$ is a point of $\PG(5, 2)$, with $P \ne R$, then $|N_{\mathcal{L}}(P) \cap N_{\mathcal{L}}(R)| = 1$, if $d(P, R) \in \{2, 4\}$, whereas $|N_{\mathcal{L}}(P) \cap N_{\mathcal{L}}(R)| = 0$, if $d(P, R) = 6$, where $d$ is the distance in the incidence graph of $\mathcal{H}(2)$. Hence $\mathcal{L}$ is $q$-Deza graph with parameters $(6, 2, 1, 0; 2)$. The stabilizer of $\mathcal{L}$ in $\PGL(6, 2)$ is the Chevalley group $G_2(2)$.
\end{example}

\begin{example}\label{ex:singer}
Let $N$ be the normalizer in $\PGL(6, 2)$ of a Singer group $H$ of $\PG(5, 2)$. Then $N$ is a group of order 378 and contains three conjugacy classes of subgroups of order $63$, one of them the conjugacy class of $H$. With the aid of a computer algebra package Magma (\cite{magma}) it can easily be checked that the subgroups of exactly one of the two remaining classes have a transitive action on points of $\PG(5, 2)$. Let $K$ be such a subgroup. In its action on lines of $\PG(5, 2)$, the group $K$ has 10 orbits of size 63 and one orbit of size 21. Among these 10 line orbits, there are precisely three such that each of them is the edge set of a $q$-Deza graph with parameters $(6, 2, 1, 0; 2)$ that is not strongly regular. These three $q$-Deza graphs are projectively equivalent, since they are permuted in a single orbit by the group $N$. The stabilizer of a $q$-Deza graph arising in this way in $\PGL(6, 2)$ coincides with $K$. For the convenience of the readers, more details are given below. We shall find it helpful to represent the elements of $\PGL(6, 2)$ as matrices in $\GL(6, 2)$. We shall consider the points as row vectors, with matrices acting on the right. Up to conjugacy, we can consider $H=\langle\sigma\rangle$ and $K=\langle\sigma^3,\varphi\rangle$, where
\[
    \sigma=
    \begin{pmatrix}
        0&1&0&0&0&0\\
        0&0&1&0&0&0\\
        0&0&0&1&0&0\\
        0&0&0&0&1&0\\
        0&0&0&0&0&1\\
        1&1&0&1&1&0
    \end{pmatrix}\quad\text{ and }
    \quad\varphi=
    \begin{pmatrix}
        0&0&0&1&0&1\\
        1&1&0&0&0&0\\
        0&0&0&0&1&1\\
        0&1&0&0&0&0\\
        0&0&0&0&0&1\\
        1&0&1&0&1&1
    \end{pmatrix}
\]
Then, a set of representatives of the three line orbits that give rise to a $q$-Deza graph is given by
\begin{align*}
    \left\{ \langle(1,0,0,0,0,0),(0,1,1,0,0,0)\rangle, \langle(1,0,0,0,0,0),(0,0,1,1,0,0) \rangle, \right. \\
    \left. \langle (0,1,0,0,0,0),(0,0,1,1,0,0)\rangle \right\}.
\end{align*}
\end{example}

\begin{lemma}\label{lem:dezanumbers}
	Let $\E$ be a $q$-Deza graph with parameters $(v,k,\lambda_{1},\lambda_{2};q)$. For any $\lambda\in\{\lambda_{1},\lambda_{2}\}$ then
	\begin{itemize}
		\item any two vertices that have $\lambda$ common neighbors are adjacent, or
		\item any two vertices that have $\lambda$ common neighbors are not adjacent, or
		\item $\lambda=1$ and $q=2$.
	\end{itemize}
\end{lemma}
\begin{proof}
	We assume neither the first nor the second possibility in the statement of the lemma is fulfilled. So, there are non-adjacent vertices $X,X'$, with $X'\neq X$, having $\lambda$ common neighbors, and there are adjacent vertices $Y,Y'$ having $\lambda$ common neighbors.
	\par Since $N_{\E}(X)$ and $N_{\E}(X')$ are subspaces, also $N_{\E}(X)\cap N_{\E}(X')$ is a subspace. Note that $N_{\E}(X)\cap N_{\E}(X')$ contains the 2-subspace $\langle X,X'\rangle$. It follows that $\lambda=\frac{q^{a}-1}{q-1}-2$ for some integer $a\geq2$. Since $N_{\E}(Y)$ and $N_{\E}(Y')$ are subspaces, also $N_{\E}(Y)\cap N_{\E}(Y')$ is a subspace. Note that $N_{\E}(Y)\cap N_{\E}(Y')$ contains neither $Y$ nor $Y'$. It follows that $\lambda_{1}=\frac{q^{b}-1}{q-1}$ for some integer $b\geq0$. We find that
	\begin{align*}
	\frac{q^{a}-1}{q-1}-2=\frac{q^{b}-1}{q-1} \quad\Leftrightarrow\quad q^{a}-q^{b}=2(q-1),
	\end{align*}
	which implies that $a>b$. We rewrite the previous expression once more and find that
	\begin{align*}
	q^{b}\frac{q^{a-b}-1}{q-1}=2.
	\end{align*}
	It follows that $a=2$, $b=1$ and $q=2$. Hence, $\lambda=1$.
\end{proof}

\begin{thm}\label{th:dezaclassification}
	If $\E$ be a q-Deza graph with parameters $(v,k,\lambda_{1},\lambda_{2};q)$, then
	\begin{itemize}
		\item $\E$ is strongly regular, or
		\item $k$ is even, and $(v,k,\lambda_{1},\lambda_{2};q)=(2k+1,k,2^{a}-1,1;2)$ for some integer $a$ with $2 \leq a\leq k-1$, where $(a-1) \mid k$ or $(a-1) \mid (k-2)$, or
		\item $vk$ is even, and $(v,k,\lambda_{1},\lambda_{2};q)=(v,k,1,0;2)$ with $v\geq 2k+2$ and $k\geq2$.
	\end{itemize}
    In the second and third case any two adjacent vertices have precisely one common neighbor.
\end{thm}
\begin{proof}
	Given Lemma \ref{lem:dezanumbers} we distinguish between six cases. If any two vertices that have $\lambda_{1}$ or $\lambda_{2}$ common neighbors are adjacent, then $\E$ is a complete graph, and thus strongly regular. If any two vertices that have $\lambda_{1}$ or $\lambda_{2}$ common neighbors are not adjacent, then $\E$ is an empty graph, and thus strongly regular. If any two vertices that have $\lambda_{1}$ common neighbors are adjacent, and any two vertices that have $\lambda_{2}$ common neighbors are not adjacent, then any two adjacent vertices have $\lambda_{1}$ common neighbors and any two non-adjacent vertices have $\lambda_{2}$ common neighbors, so $\E$ is strongly regular. If $\lambda_{1}=\lambda_{2}=1$, then $\E$ is clearly also strongly regular. Now, we look at the two remaining cases. We may assume that $\lambda_{1}\neq\lambda_{2}$ and that there are both pairs of vertices having $\lambda_{1}$ common neighbors and pairs of vertices having $\lambda_{2}$ common neighbors, since otherwise the graph is clearly strongly regular.
	\par Assume that any two vertices that have $\lambda_{1}\neq 1$ common neighbors are adjacent, and that $\lambda_{2}=1$ and $q=2$. Let $X$ and $X'$ be different vertices having $\lambda_{1}$ common neighbors. Since $N_{\E}(X)$ and $N_{\E}(X')$ are subspaces, also $N_{\E}(X)\cap N_{\E}(X')$ is a subspace. Note that $N_{\E}(X)\cap N_{\E}(X')$ contains the 2-subspace $\langle X,X'\rangle$. It follows that $\lambda_{1}=(2^{a}-1)-2=2^{a}-3$ for some  integer $a\geq2$. Since $\lambda_{1}\neq1$, we even have $a\geq3$. Furthermore, $a\leq k+1$ since all common neighbors of two vertices are contained in the neighborhood of one vertex. From Theorem \ref{th:dezastandardequation} it follows that
	\begin{align*}
		n_{\E}(\lambda_{1})=2\frac{\left(2^{k+1}-3\right)\left(2^{k}-1\right)-\left(2^{v-1}-1\right)}{(2^{a}-4)}=\frac{\left(2^{k+1}-1\right)\left(2^{k-1}-1\right)-\left(2^{v-2}-1\right)}{2^{a-2}-1}\;.
	\end{align*}
	Now, since $n_{\E}(\lambda_{1})\geq0$ we have $v-2<2k$, so $v\leq 2k+1$. However, if $v\leq 2k$ then the intersection of any two $(k+1)$-spaces has dimension at least 2, contradicting that there are (non-adjacent) vertices with $\lambda_{2}=1$ common neighbor. So, $v=2k+1$. Hence,
	\begin{align*}
		n_{\E}(\lambda_{1})=2\frac{\left(2^{k}-1\right)\left(2^{k-2}-1\right)}{2^{a-2}-1}\;.
	\end{align*}
    Fix an arbitrary vertex $X$. Since all vertices that have $\lambda_{1}$ common neighbors with $X$ are adjacent to $X$, it follows that $n_{\E}(\lambda_{1})\leq|N_{\E}(X)|=2^{k+1}-2$. This implies that
	\begin{align}\label{eq:deza5}
		2^{k-2}-1\leq 2^{a-2}-1\;.
	\end{align}
	Since $a\leq k+1$, we have that $a\in\{k,k+1\}$. If $a=k$, then we have equality in Equation~\eqref{eq:deza5}, so any neighbor of $X$ has $\lambda_{1}$ common neighbors with $X$. As $X$ was arbitrary it follows that two vertices are adjacent if and only if they have $\lambda_{1}$ common neighbors. So, $\E$ is strongly regular.
	\par If $a=k+1$, then $n_{\E}(\lambda_{1})=2^{k}-3-\frac{1}{2^{k-1}-1}$, so $k=2$ as $n_{\E}(\lambda_{1})$ must be an integer. So, $\E$ has parameters $(5,2,2^{3}-3,1;2)$. Let $Y$ and $Z$ be two vertices that have $\lambda_{1}$ common neighbors. By assumption $Y$ and $Z$ are adjacent. The subspaces $N_{\E}(Y)$ and $N_{\E}(Z)$ contain $2^{3}-1=7$ vector lines (vertices) but $Y$ and $Z$ have 5 common neighbors, so $N_{\E}(Y)=N_{\E}(Z)$. We denote this 3-space in $\F^{5}_{2}$ by $\sigma$. Any vertex $V$ in $\sigma\setminus\langle Y,Z\rangle$ is adjacent to both $Y$ and $Z$, so $N_{\E}(V)=\sigma$. Repeating the argument with two other vertices $N_{\E}(V)=\sigma$ for any vertex $V\in\langle Y,Z\rangle$. Now, let $W$ be a vertex not in $\sigma$. Then $Y$ and $W$ are not adjacent, hence have $\lambda_{2}=1$ common neighbor, say $W'$. However, then $Y\in N_{\E}(W')=\sigma$, a contradiction. This concludes this case.
	\par Assume now that any two vertices that have $\lambda_{1}\neq 1$ common neighbors are not adjacent, and that $\lambda_{2}=1$ and $q=2$. Let $X$ and $X'$ be different vertices having $\lambda_{1}$ common neighbors. Since $N_{\E}(X)$ and $N_{\E}(X')$ are subspaces, also $N_{\E}(X)\cap N_{\E}(X')$ is a subspace. Note that $N_{\E}(X)\cap N_{\E}(X')$ contains neither $X$ nor $X'$. It follows that $\lambda_{1}=2^{a}-1$ for some integer $a\leq k$. Since $\lambda_{1}\neq1$, we have $a\neq 1$. 
	From Theorem \ref{th:dezastandardequation} it follows that
	\begin{align*}
		n_{\E}(\lambda_{1})=2\frac{\left(2^{k+1}-3\right)\left(2^{k}-1\right)-\left(2^{v-1}-1\right)}{(2^{a}-2)}=2\frac{\left(2^{k+1}-1\right)\left(2^{k-1}-1\right)-\left(2^{v-2}-1\right)}{2^{a-1}-1}\;.
	\end{align*}
	Now, since $n_{\E}(\lambda_{1})\geq0$ we have that either $a\geq2$ and $v-2<2k$, so $v\leq 2k+1$, or that $a=0$ and $v-2\geq2k$. In the latter case the number of edges of $\E$ equals $\frac{\left(2^v-1\right)\left(2^{k}-1\right)}{3}$, by Lemma \ref{lem:numberedges}. Hence, $v k$ must be even and we find the third possibility in the statement of the lemma. Note that for $k=1$ the $q$-graph $\E$ would be strongly regular. We look now at the former case.
	\par Assume that $a\geq2$ and $v\leq 2k+1$. We know that any two adjacent vertices have $\lambda_{2}=1$ common neighbor. If any two non-adjacent vertices have $\lambda_{1}$ common neighbors the graph is strongly regular, so we can assume without loss of generality that there are two non-adjacent vertices having $\lambda_{2}=1$ common neighbor. If $v\leq 2k$ then the intersection of any two $(k+1)$-spaces has dimension at least 2, contradicting that there are non-adjacent vertices with $\lambda_{2}=1$ common neighbor. So, $v=2k+1$. It follows from Lemma \ref{lem:numberedges} that the number of edges of $\E$ equals $\frac{\left(2^{2k+1}-1\right)\left(2^{k}-1\right)}{3}$. Hence, $k$ must be even.
	We also find that
	\begin{align*}
		n_{\E}(\lambda_{1})=2\frac{\left(2^{k+1}-1\right)\left(2^{k-1}-1\right)-\left(2^{2k-1}-1\right)}{2^{a-1}-1}=4\frac{\left(2^{k}-1\right)\left(2^{k-2}-1\right)}{2^{a-1}-1}\;.
	\end{align*}
	We see that if $k=2$, then $n_{1}(\lambda_{1})=0$, and then $\E$ is strongly regular. So, we can assume $k\geq4$. Also, if $a=k$, then
	\begin{align*}
	    n_{\E}(\lambda_{1})=4\frac{\left(2^{k}-1\right)\left(2^{k-2}-1\right)}{2^{k-1}-1}=8\left(2^{k-2}-1\right)+4\cdot \frac{2^{k-2}-1}{2^{k-1}-1}\notin\N\:,
	\end{align*}
	so $a\leq k-1$. Let $2 \le a \le k-1$ be such that $(a-1) \nmid k$ and $(a-1) \nmid (k-2)$. Define $t = \gcd(k, a-1)$. From $n_{\E}(\lambda_{1})\in\N$, $\gcd(4, 2^{a-1}-1) = 1$ and $\gcd(2^s-1, 2^{a-1}-1) = 2^{\gcd(s,a-1)}-1$ for all $s\in\N$, it follows that $(2^{a-1}-1) \mid \left(2^{t}-1\right)\left(2^{\gcd(k-2,a-1)}-1\right)$. If $t=1$, we have that $(a-1)\mid\gcd(k-2,a-1)$ and thus $(a-1)\mid(k-2)$, a contradiction. If $t >1$, then $\gcd(k-2, a-1)$ equals $1$ or $2$. In the former case, we find that $(a-1)\mid t\mid k$, a contradiction. If the latter case occurs, then $a$ is odd and the equality $\frac{3}{c} = \frac{2^{a-1}-1}{2^t-1}\in\N$ holds true, for some integer $c$. It turns out that either $c = 3$ and $a-1=t$, and thus $(a-1) \mid k$, a contradiction, or $c = 1$ and $2^t-1 = \frac{2^{a-1}-1}{3}$, where $2^{t}-1 \equiv -1 \pmod{4}$ and $\frac{2^{a-1}-1}{3} \equiv{1} \pmod{4}$, a contradiction. We find the second example in the statement of the theorem.
\end{proof}

For the third class of $q$-Deza graphs given by Theorem~\ref{th:dezaclassification}, the following result holds. 

\begin{prop}\label{prop:extendthroughspread}
If there exists a $q$-Deza graph with parameters $(v, k, 1, 0; 2)$, then there exists a $q$-Deza graph with parameters $(vt, k, 1, 0; 2)$ for every integer $t\geq2$. 
\end{prop}
\begin{proof}
Let $\mathcal{D}$ be a $(v-1)$-spread of $\PG(vt-1, 2)$ and consider a $q$-Deza graph with parameters $(v, k, 1, 0; 2)$ in each of the $\frac{2^{vt}-1}{2^{v}-1}$ members of $\mathcal{D}$. It is easily seen that their union is a $q$-Deza graph with parameters $(vt, k, 1, 0; 2)$.  
\end{proof}
Recall that there are examples of $q$-Deza graphs with parameters $(6,2,1,0;2)$, see Example \ref{ex:DezasplitCayleyhexagon} and Example \ref{ex:singer}. Note that the $q$-Deza graph from Proposition \ref{prop:extendthroughspread} is disconnected. 
\begin{cor}
There exists $q$-Deza graphs with parameters $(6t, 2, 1, 0; 2)$ for all positive integers $t$.
\end{cor}

\subsection{The \texorpdfstring{$q$}{q}-Deza graphs with parameters \texorpdfstring{$(6,2,1,0;2)$}{(6,2,1,0;2)}}

Let $\E$ be a $q$-Deza graph with parameters $(6, 2, 1, 0; 2)$. In this subsection, we will prove that $\E$ arises from Example~\ref{ex:DezasplitCayleyhexagon} or Example~\ref{ex:singer}.

We know that $\E$ consists of $63$ lines of $\PG(5, 2)$, such that through a point $P \in \PG(5, 2)$ there pass $3$ lines of $\E$ and these lines are in a plane, which we will denote by $\pi_P$. Moreover, $\pi_P \cap \pi_Q =  P Q \in \E$, if $Q \in \pi_P \setminus \{P\}$, whereas there are $24$ points $Q \in \PG(5, 2) \setminus \pi_P$ such that $|\pi_Q \cap \pi_P| = 1$ and the remaining $32$ points $Q$ of $\PG(5, 2)$ are such that $|\pi_Q \cap \pi_P| = 0$. It is easily seen that 
\begin{align}
P \in \pi_Q \iff Q \in \pi_P. \label{p}
\end{align}

\begin{lemma}\label{l1}
Let $H$ be a hyperplane of $\PG(5, 2)$ and let $\E_H$ be the set of lines of $\E$ contained in $H$. Then $|\E_H| = 15$ and through a point of $H$ there pass either $1$ or $3$ lines of $\E_H$. Moreover, there are precisely $7$ points of $H$ that are incident with $3$ lines of $\E_H$. 
\end{lemma}
\begin{proof}
The lines of $\E\setminus\E_{H}$ meet $H$ in a point. Counting the pairs $(P,\ell)$ with $\ell\in\E_{H}$ and $P$ a point in $H$ on $\ell$, we find $3|\E_{H}|+(63-|\E_{H}|)=3 \times 31$. So, there are precisely $15$ lines of $\E$ contained in a hyperplane $H$. Note that if $P \in H$, then either $\pi_P \cap H$ is a line of $\E_H$ or $\pi_P \subset H$. Hence there is $1$ or $3$ lines of $\E_H$ through $P$. Let $x_i$ be the number of points of $H$ incident with $i$ lines of $\E_H$. Then $x_1 + x_3 = 31$ and $x_1 + 3x_3 = 45$, where the last equality is obtained by counting the pairs $(P, \ell)$, $P \in H$, $\ell \in \E_H$, $P \in \ell$. The statement follows.   
\end{proof}

Let $\S_H$ denote the set consisting of the $7$ points of $H$ incident with $3$ lines of $\E_H$.

\begin{example}\label{badex}
Let $H$ be a hyperplane in $\PG(5,2)$, let $P_1,P_2,P_3,P_4,P_5$ be five points that span $H$, and set $P_6=P_1+P_2+P_5$ and $P_7=P_3+P_4$. Let 
\begin{align*}
\cQ &= \{P_1, P_2, P_3, P_4, P_5, P_6, P_7\}, \\
\cR &= \{ P_1 P_2, \; P_2 P_3, \; P_3 P_4, \; P_4 P_5, \; P_5 P_1, \; P_1 P_6, \; P_6 P_7, \; P_2 (P_1+P_3), \; P_3 (P_2+P_4), \\ 
&\qquad P_7 (P_3+P_6), P_4 (P_3+P_5), \; P_5 (P_1+P_4), \; P_6 (P_1+P_7), \\
&\qquad (P_1 + P_2 + P_4) (P_2+P_3+P_5), \; (P_1 + P_3+P_5) (P_2+P_4+P_5)\}.
\end{align*}
Then $\cR$ consists of $15$ lines of $H$ such that through a point of $\cQ$ there pass three coplanar lines of $\cR$, whereas a point of $H \setminus \cQ$ is incident with precisely one line of $\cR$. Note that $P_3$, $P_4$ and $P_7$ are collinear; let $\ell$ be the line consisting of these three points. The four points $P_1, P_2, P_5, P_6$ are coplanar; let $\gamma$ be the plane containing them. Note that $\ell$ and $\gamma$ are disjoint. In particular $P_1, P_2, P_5, P_6$ form the complement of a line $m$ in $\gamma$. There are 13 lines in $\cR$ that contain at least one of the points in $\cQ$. These 13 lines cover in total 25 points. The six remaining points of $H$ form two lines, which are the two final lines of $\cR$. The stabilizer of $\cR$ in $\PGL(5, 2)$ is a group of order $6$, isomorphic to $S_3$. 
\end{example}

\begin{rem}\label{prop-badex}
Let $\cQ$, $\cR$, $\ell$, $m$ and $\gamma$ be as described in Example~\ref{badex}. There are three solids in $H$ containing exactly one line of $\cR$, namely $\langle m, P_3, P_2+P_4 \rangle$, $\langle m, P_4, P_3+P_5 \rangle$ and $\langle m, P_7, P_3+P_6\rangle$, and one point of $\cQ$ on that line. These are the three solids spanned by $m$ containing neither $\gamma$ nor $\ell$. Each of them meets $\ell$ in precisely one point and is disjoint to $\gamma \cap \cQ$, see also Lemma~\ref{l3}. Note that this unique line of $\cR$ in such a solid contains a unique point of $\cQ$. Moreover, these three solids are projectively equivalent under the stabilizer of $\cR$ in $\PGL(5, 2)$.
\par Since $|\cR^{\PGL(5, 2)}| = |\PGL(5, 2)|/6 = 1666560$, a double counting argument on the tuples $(P, r, \Sigma, \mathcal{L})$, where $\mathcal{L} \in \cR^{\PGL(5, 2)}$, $\Sigma$ is a solid of $H$ containing exactly the line $r$ of $\mathcal{L}$ and $P$ is the point of $\cQ$ on $r$, gives 
\begin{align*}
    1666650 \times 3 = z \times 31 \times 35 \times 3.
\end{align*}
Here $z$ denotes the number of elements of $\cR^{\PGL(5, 2)}$ having precisely one line in a fixed but arbitrary $\Sigma$. Therefore $z = 1536$. 
\end{rem}

\begin{lemma}\label{l2}
If $\S_H \not\subset \pi_P$ for all $P \in \S_H$, then $|\S_H \cap \pi_T| \in \{3, 4\}$ for any $T\in\S_{H}$ and both cases occur.
\end{lemma}
\begin{proof}
    Consider a point $T\in\S_H$. For every point $Q\in\S_H\setminus\pi_T$, we know that $\pi_T\cap\pi_Q$ is a point distinct from $T$, and this point is necessarily contained in $\S_H$. So, we have that $|\S_H\cap\pi_T|\geq2$, and that
    \begin{align}\label{eq:Shplane}
        \S_H\subseteq\bigcup_{X\in(\S_H\cap\pi_T)\setminus\{T\}}\pi_{X}\:.
    \end{align}
    \par If $|\S_H\cap\pi_T|=2$, let $R$ be the point such that $\S_H\cap\pi_T=\{T,R\}$. Then by \eqref{eq:Shplane}, we have that $\S_{H}\subseteq\pi_{R}$, contradicting the assumption. So, we may assume that $|\S_H\cap\pi_T|\geq3$ for all $T\in\S_H$.
    \par If $|\S_H\cap\pi_T|\in\{5,6\}$, let $Q$ be a point in $\S_H \setminus \pi_T$, and let $Q'\in\S_H$ be the unique point in $\pi_{T}\cap\pi_{Q}$. There are at least two points $X_1$ and $X_2$ in $\pi_{T}\cap\S_H$ not on the line $TQ'$. For any point $X_{i}$, we know that the unique point $X'_{i}$ in $\pi_{X_{i}}\cap\pi_{Q}$ belongs to $\S_H$. Moreover $X'_1\neq X'_2$ since $\pi_{X'_i}\cap\pi_T$ is a unique point, namely $X_i$. For the same reason $X'_{1}\neq Q\neq X'_{2}$, but this means there are 3 points of $\S_{H}$ not in $\pi_T$, a contradiction. So, we may assume that $|\S_H\cap\pi_T|\leq4$ for all $T\in\S_H$.
    \par Suppose that $\S_H \cap \pi_T = \{T, R, U\}$. By \eqref{eq:Shplane} we get that $\S_H \in \pi_{R} \cup \pi_U$. Therefore $|\S_H \cap \pi_R| = |\S_H \cap \pi_U| = 4$.
    \par Suppose that $\S_H \cap \pi_T= \{T, R, U, V\}$, and also that there is no line of $\pi_T$ through $T$ all whose points are contained in $\S_H$. A point of $\S_H\setminus\pi_T$ can be in at most one of the planes $\pi_R,\pi_U,\pi_V$. So, $|\S_H \cap \pi_R| = |\S_H \cap \pi_{U_1}| = |\S_H \cap \pi_{U_2}| = 3$, using \eqref{eq:Shplane}.
    \par The only case that is left, is where $|\S_H \cap \pi_T|=4$ for all $T\in\S_H$ and such that for all $T\in\S_H$ there is a line through $T$ in $\pi_T$ containing three points of $\S_H$. But this is impossible since 3 does not divide 7.
\end{proof}

\begin{lemma}\label{l3}
A solid in $H$ contains as many points of $\S_H$ as lines of $\E_H$. 
\end{lemma}
\begin{proof}
Let $\Sigma$ be a solid in $H$ containing $x$ points of $\S_H$ and $y$ lines of $\E_H$. By double counting the pairs $(P, \ell)$, $P \in \Sigma$, $\ell \in \E_H$, $P \in \ell$, we get $3x + (15-x) = 3y + (15-y)$, as required.  
\end{proof}

\begin{lemma}\label{l4}
If $\S_H \not\subset \pi_P$ for all $P \in \S_H$, then $\E_H$ is projectively equivalent to Example~\ref{badex}.
\end{lemma}
\begin{proof}
Let $\S_H = \{P_i \mid i = 1, \dots, 7\}$ be such that $\langle P_1, \dots, P_7 \rangle$ is not a plane. Then, by Lemma~\ref{l2}, there is a plane, say $\pi_{P_2}$ such that $|\S_H \cap \pi_{P_2}| = \{P_1, P_2, P_3\}$. Since $\pi_{P_i} \cap \pi_{P_2} \in \{P_1, P_3\}$, if $i = 4, \dots, 7$, and \eqref{p} holds true, we get that $\S_H \subset (\pi_{P_1} \cup \pi_{P_3})$, and therefore $|\S_H \cap \pi_{P_1}| = |\S_H \cap \pi_{P_3}| = 4$ and $P_3\notin P_1P_2$. Let $P_4, P_7 \in \pi_{P_3} \setminus \pi_{P_2}$ and $P_5, P_6 \in \pi_{P_1} \setminus \pi_{P_2}$ be such that $P_4 \in \pi_{P_5}$ and $P_7 \in \pi_{P_6}$. If both $\{P_1, P_5, P_6\}$ and $\{P_3, P_4, P_7\}$ were lines, then $\pi_{P_6} \cap \pi_{P_4} = P_5 P_7$. It follows that $P_5 P_7 \in \E_H$, but the point $P_7$ is not contained in $\pi_{P_5}$, because the four points in it are $P_1$, $P_4$, $P_5$ and $P_6$, a contradiction. If the points $P_1$, $P_5$ and $P_6$ are not collinear, then necessarily $\{P_3, P_4, P_7\}$ is a line, otherwise $\pi_{P_6} \cap \pi_{P_4}$ would consist of a point of $H$ incident with $3$ lines of $\E_H$, and distinct from $P_1, \dots, P_7$, a contradiction. Therefore, we may assume that $\{P_1, P_5, P_6\}$ are not collinear, whereas $\{P_3, P_4, P_7\}$ is a line. Notice that $P_2$, $P_5$ and $P_6$ are not collinear, otherwise $\Sigma = \langle P_2, P_3, P_4, P_5, P_6, P_7 \rangle$ is a solid and $\pi_3, \pi_4, \pi_7$ are the three planes of $\Sigma$ through the line $P_3 P_4$. In particular $\Sigma$ contains at least $7$ lines of $\E_H$ and precisely $6$ points of $\S_H$, contradicting Lemma~\ref{l3}. It follows that $P_1, P_2, P_5, P_6$ form the complement of a line in a plane $\gamma$ and $\ell = \{P_3, P_4, P_7\}$ is a line disjoint from $\gamma$. 
The configuration of the planes $\pi_{P_1},\dots,\pi_{P_7}$ as in Example~\ref{badex} arises. As explained there, this determines also the final two lines of $\E_H$ uniquely. So, we indeed find Example~\ref{badex}.
\end{proof}

\begin{prop}
A $q$-Deza graph $\E$ with parameters $(6, 2, 1, 0; 2)$ arises from Example~\ref{ex:DezasplitCayleyhexagon} or Example~\ref{ex:singer}. 
\end{prop}
\begin{proof}
Assume that for every hyperplane $H$, there exists a point $P \in \S_H$ such that $\S_H = \pi_P$. Then $H \setminus \pi_P$ consists of the $24$ points $Q \in \PG(5, 2) \setminus \pi_P$ such that $|\pi_Q \cap \pi_P| = 1$ and through a point of $H \setminus \pi_P$ passes exactly one line of $\E_H$. Consider the incidence structure $(\mathcal{P}, \E)$, where $\mathcal{P}$ is the set of points of $\PG(5, 2)$, with incidence the natural one. We claim that $(\mathcal{P}, \E)$ is a generalized hexagon. Let $\Gamma$ be the (simple, bipartite) incidence graph of $(\mathcal{P}, \E)$. The graph $\Gamma$ is 3-regular. Since two points determine a unique line, there are no cycles of length 4 in $\Gamma$. Also, there are no lines of $\E$ forming a triangle or a quadrangle, hence there are no cycles of length $6$ or $8$ in $\Gamma$. So the girth of $\Gamma$ is at least 10. If there were a cycle of length 10 in $\Gamma$, then we would find a pentagon consisting of lines of $\E$. Such a pentagon has to span a hyperplane $H$, where $\S_H = \pi_P$, for some $P \in \S_H$. Then we find a vertex of the pentagon $Q \in H \setminus \pi_P$ such that there are two lines of $\E_H$ through $Q$, a contradiction. Finally, it can be easily checked that $\Gamma$ contains cycles of length $12$ and the result follows from \cite{TV}.

Assume that there exists a hyperplane $H$ such that $\S_H \not\subset \pi_P$, for all $P \in \S_H$. Then by Lemma~\ref{l4} we may assume that $\S_H$ and $\E_H$ are as in Example~\ref{badex}. Let $\Sigma$ be one of the three solids of $H$ containing precisely one line of $\E_{H}$, say $r$, see Remark~\ref{prop-badex}. Then $\Sigma \cap \S_H$ is a point $R$ belonging to $r$. Let $H'$ and $H''$ be the other two hyperplanes containing $\Sigma$. Then $r$ is the unique line of $\E_{H'}$ (and of $\E_{H''}$) contained in $\Sigma$. Therefore $\Sigma \cap \S_{H'}$ (and $\Sigma \cap \S_{H''}$) consists of one point belonging to $r \setminus \{R\}$. Note that necessarily $r = \{\Sigma \cap \S_H, \Sigma \cap \S_{H'}, \Sigma \cap \S_{H''}\}$. If $\S_{H'} = \pi_P$, for some $P \in \S_{H'}$, then $\Sigma \cap \pi_P = \Sigma \cap \S_{H'}$, where $|\Sigma \cap \S_{H'}| = 3$, a contradiction. Hence $\S_{H'} \not\subset \pi_P$, for all $P \in \S_{H'}$. Similarly $\S_{H''} \not\subset \pi_P$, for all $P \in \S_{H''}$. By Lemma~\ref{l4}, it follows that both $(\S_{H'}, \E_{H'})$, $(\S_{H''}, \E_{H''})$ are projectively equivalent to $(\S_H, \E_H)$ as described in Example~\ref{badex}. By Remark~\ref{prop-badex}, there are $2 \times 1536^2$ couples $\left\{(\S_{H'}, \E_{H'}), (\S_{H''}, \E_{H''})\right\}$ such that 

$r$ is the unique line of $\E_{H'}$ in $\Sigma$, the unique line of $\E_{H''}$ in $\Sigma$, and such that $r = \{\Sigma \cap \S_H, \Sigma \cap \S_{H'}, \Sigma \cap \S_{H''}\}$. With the aid of Magma~\cite{magma}, we determined that only $1120$ among these couples are such that for each point of $\Sigma \setminus r$ the three lines of $\E_{H} \cup \E_{H'} \cup \E_{H''}$ through it are coplanar. For a fixed couple $\left\{(\S_{H'}, \E_{H'}), (\S_{H''}, \E_{H''})\right\}$, let $\mathcal{X}$ denote the set $\Sigma \cup \S_{H} \cup \S_{H'} \cup \S_{H''}$, where we interpret $\Sigma$ as its point set. Then $\E_{H} \cup \E_{H'} \cup \E_{H''}$ is a set of $43$ lines such that through a point $Q$ there pass one or three (coplanar) lines of $\E_{H} \cup \E_{H'} \cup \E_{H''}$, according as $Q \in \PG(5, 2) \setminus \mathcal{X}$ or $Q \in \mathcal{X}$, respectively. Hence, for a line $m$ disjoint from $\mathcal{X}$, we find three planar pencils determined by $(m, n)$, $(m, n')$ and $(m, n'')$, where $n$, $n'$ and $n''$ are the unique lines of $\E_{H}$, $\E_{H'}$ and $\E_{H''}$ incident with $m$, respectively. The line $m$ is said to be {\em good} with respect to the couple $\{(\S_{H'}, \E_{H'}),(\S_{H''}, \E_{H''})\}$, if for each of the three planar pencils determined by $(m, n)$, $(m, n')$ and $(m, n'')$, its third line is disjoint from $\mathcal{X}$. Some computations performed with Magma show that the previous mentioned couples $\left\{(\S_{H'}, \E_{H'}), (\S_{H''}, \E_{H''})\right\}$ have at most $20$ good lines and in case of equality the $20$ good lines together with $\E_{H} \cup \E_{H'} \cup \E_{H''}$ give rise to a line-set projectively equivalent to Example~\ref{ex:singer}.
\end{proof}

\section*{Acknowledgement}

This work has been partially supported by Croatian Science Foundation under the projects 4571 and 5713.


\begin{thebibliography}{SK}

\bibitem{spectra-deza}
S. Akbari, A. H. Ghodrati, M. A. Hosseinzadeh, V. V. Kabanov, E. V. Konstantinova, L. V. Shalaginov, Spectra of Deza graphs, {\em Linear Multilinear Algebra}, 70 (2022), 310--321. 

\bibitem{magma} W. Bosma, J. Cannon, C. Playoust, The Magma algebra system I: the user language, {\em J. Symb. Comput.}, 24 (1997), 235--265.

\bibitem{qsrg}
M. Braun, D. Crnkovi\' c, M. De Boeck, V. Mikuli\' c Crnkovi\' c, A. \v Svob, $q$-Analogs of strongly regular graphs, {\em Linear Algebra Appl.}, 693 (2024), 362-373. 

\bibitem{mario-subspace}
M.~Braun,~M.~Kiermaier,~A.~Wassermann, q-Analogs of Designs: Subspace Designs, in M.~Greferath,~M.~O.~Pav\v cevi\' c~N.~Silberstein,~M.~\'{A}ngeles~V\'{a}zquez-Castro (Eds.), Network Coding and Subspace Designs, 171--211. Springer, 2018.


\bibitem{buratti}
M.~Buratti, A.~Naki\' c, Designs over finite fields by difference methods, {\em Finite Fields Appl.}, 57 (2019), 128--138.

\bibitem{buratti2}
M.~Buratti, A.~Naki\' c, A. Wassermann, Graph decompositions in projective geometries, {\em J. Combin. Des.}, 29 (2021), 141--174.
%
%
%
%
\bibitem{cohentits} 
A.~M.~Cohen, J.~Tits, On Generalized Hexagons and a Near Octagon whose Lines have Three Points, {\em Europ. J. Combinatorics}, 6 (1985), 13--27.


%
%
\bibitem{delsarte}
P.~Delsarte, Association schemes and t-designs in regular semilattices, {\em J. Comb. Theory Ser. A}, 20 (1976), 230--243.

%


\bibitem{deza}
M.~Erickson, S.~Fernando, W.~H.~Haemers, D.~Hardy, J.~Hemmeter, Deza graphs: A generalization of strongly regular graphs, {\em J. Combin. Des.}, 7 (1999), 395--405.
%
%
\bibitem{Deza-Haemers}
S.~Goryainov, W.~H.~Haemers, V.~Kabanov, L.~Shalaginov, Deza graphs with parameters $(n,k,k-1,a)$ and $\beta=1$, {\em J. Combin. Des.}, 27 (2019), 188--202.

\bibitem{Deza-Goryainov}
S.~Goryainov, D.~Panasenko, On vertex connectivity of Deza graphs with parameters of the complements to Seidel graphs, {\em European J. Combin.}, 80 (2019), 143--150. 


\bibitem{ddg} W.~H.~Haemers, H.~Kharaghani, M.~Meulenberg, Divisible design graphs, {\em J. Combin. Theory Ser. A}, 118 (2011), 978--992.

\bibitem{spyros} M.~R.~Hurley,~B.~K.~Khadka,~S.~S.~Magliveras, Some new large sets of geometric designs of type ${\rm LS}[3][2,3,2^8]$, {\em J. Algebra Comb. Discrete Struct. Appl.}, \textbf{3} (2016), 165--176.

\bibitem{Deza-Kabanov}
V.~Kabanov, N.~V.~Maslova, L.~V.~Shalaginov, On strictly Deza graphs with parameters $(n,k,k-1,a)$, {\em European J. Combin.}, 80 (2019), 194--202.

%
%
%
%
%
\bibitem{rudvalis}
A.~Rudvalis,
$(v,k, \lambda )$-graphs and polarities of $(v,k, \lambda )$-designs,
Math. Z. 120 (1971), 224--230.
%
%
\bibitem{segre}
B.~Segre, Teoria di Galois, fibrazioni proiettive e geometrie non desarguesiane, {\em Ann. Mat. Pura Appl.}, (4)  64 (1964), 1--76.
%

\bibitem{TV} J.~A.~Thas, H.~Van Maldeghem, Embedded thick finite generalized hexagons in projective spaces, {\em J. London Math. Soc.}, (2) 54 (1996), 566--580.

\bibitem{bookvanmaldeghem}
H.~Van Maldeghem, Generalized Polygons, Monographs in Mathematics, Birkhäuser Basel (1998).
\end{thebibliography}
\end{document}